\documentclass[10pt]{article}
\RequirePackage{fix-cm}
\usepackage[table,xcdraw]{xcolor}
\usepackage[square,numbers]{natbib}

\usepackage{xcolor}
\usepackage[T1]{fontenc}    
\usepackage[breaklinks=true, colorlinks=true, linkcolor={blue!90!black}, citecolor={blue!90!black}, urlcolor={blue!90!black}]{hyperref}
\usepackage{url}            
\usepackage{booktabs}       
\usepackage{amsfonts}       
\usepackage{nicefrac}       
\usepackage{microtype}      

\usepackage{geometry}
\usepackage{graphicx}         
\usepackage[ruled,vlined,linesnumbered]{algorithm2e}

\usepackage{amsmath,amssymb,amsopn}
\usepackage{epstopdf}
\usepackage{appendix}
\usepackage{enumitem}

\usepackage{hyperref,doi}
\usepackage{multirow}
\usepackage{color}
\newcommand\comment[1]{}

\newcommand{\Rev}[1]{{\color{black}#1}}

\usepackage[normalem]{ulem} 
 
\newcommand{\E}[2][]{\operatorname{\mathbb{E}}_{#1}\left[ #2\right]} 
\newcommand{\A}[4]{\mathcal{A}\left( #1;#2;#3;#4 \right)} 
\newcommand{\B}[2]{\mathcal{B}\left( #1;#2 \right)} 
\newcommand{\Var}[2][]{\operatorname{Var}_{#1}\left[ #2\right]} 
\newcommand{\Cov}[2][]{\operatorname{Cov}_{#1}\left[ #2\right]} 
 
\newcommand{\e}{{\epsilon}} 
\newcommand{\R}{\mathbb{R}} 
\newcommand{\N}{\mathbb{N}} 
\newcommand{\LL}{\mathcal{L}} 
\newcommand{\cP}{\mathcal{P}}

\newcommand{\cD}{\mathcal{D}}
\newcommand{\cT}{\mathcal{T}}
\newcommand{\cQ}{\mathcal{Q}}
\newcommand{\cE}{\mathcal{E}}
\newcommand{\cC}{\mathcal{C}}
\newcommand{\cG}{\mathcal{G}}
\newcommand{\cH}{\mathcal{H}}

\newcommand{\iid}{i.i.d.\xspace}

\usepackage{makecell}

\newcommand{\algref}[1]{{\rm Algorithm}~\ref{alg:#1}}
\newcommand{\secref}[1]{{\rm Section}~\ref{sec:#1}}
\newcommand{\thmref}[1]{{\rm Theorem}~\ref{thm:#1}}
\newcommand{\tabref}[1]{{\rm Table}~\ref{tab:#1}}
\newcommand{\lemref}[1]{{\rm Lemma}~\ref{lem:#1}}
\newcommand{\assref}[1]{{\rm Assumption}~\ref{assm:#1}}
\newcommand{\subassref}[2]{{\rm Assumption}~\ref{assm:#1}.\ref{sub:#2}}
\newcommand{\defref}[1]{{\rm Definition}~\ref{def:#1}}

\usepackage[straightquotes]{newtxtt}

\newcommand{\minimize}{\operatornamewithlimits{minimize}}

\newcommand{\fGamma}{\mathrm{\Gamma}}

\usepackage{amsthm,authblk}
\newtheorem{theorem}{Theorem}
\theoremstyle{definition}
\newtheorem{definition}{Definition}
\newtheorem{assumption}{Assumption}

\newtheorem{lemma}{Lemma}
\begin{document}
    
    \title{Multistart Algorithm for Identifying All\\ Optima of Nonconvex Stochastic Functions
    }
    
    
  \author[1]{Prateek Jaiswal\thanks{Corresponding author: \texttt{jaiswalp@tamu.edu}}}

  \author[2]{Jeffrey Larson}

  \affil[1]{\small Department of Statistics,  Texas A\&M University, College Station, TX 77843, USA}
  \affil[2]{Mathematics and Computer Science Division, Argonne National Laboratory, Lemont, IL 60439}
    
    
    

    \maketitle

\begin{abstract}
 We propose a multistart algorithm to identify all local minima
 of a constrained, nonconvex stochastic optimization problem. The algorithm
 uniformly samples points in the domain and then starts a local stochastic
 optimization run from any point that is the ``probabilistically best'' point in
 its neighborhood.
 Under certain conditions, our algorithm is shown to asymptotically identify all 
 local optima with high probability; this holds even though our algorithm is shown to almost surely
 start only finitely many local stochastic optimization runs. 
 We demonstrate the performance of an implementation of our algorithm on nonconvex
 stochastic optimization problems, including identifying optimal variational
 parameters for the quantum approximate optimization algorithm. 
\end{abstract}

\section{Introduction} 
We consider the problem of 
identifying all local minima of
the constrained stochastic optimization problem
\begin{align}\label{eq:prob_def}\tag{SO} 
    \minimize_{x \in \mathbb{R}^d} f(x) := \E[\xi]{F(x,\xi)},~\text{subject to } x\in \cD, 
\end{align} 
where $F(x,\xi)$ is the observable random function, $\xi$ is a random variable
defined on the probability space $(\Omega, \mathbf{F},P_{\xi})$, and $\cD
\subset \R^d$ is compact. 

Identifying all local minima is relevant in various
engineering and scientific  applications. For instance, Gheribi et al.~\cite{Gheribi2011}
\Rev{describe a multicomponent chemical thermodynamics system where all low-melting
compositions---each corresponding to a local minimum of their objective---are
desired. While the system in~\cite{Gheribi2011} assumes that the evaluation of
temperatures is deterministic, in practice, such temperatures measurements
would be are accompanied by stochastic noise.} Furthermore, in biophysics and
biochemistry, an important problem is to understand the transitions in the shape of a
protein molecule as it folds itself from a disordered (high-energy) state to a
native (least-energy) state. The energy function has multiple local minima, and
each local minimum corresponds to an intermediate stable state of a protein molecule.
Identifying each of them can be crucial for understanding the protein folding
pathways~\cite{adamcik2018amyloid, floudas1999global}. \Rev{Li et al.~\citep{Zhenqin1987} highlight the
stochastic nature of the protein folding process.}
Moreover, as a direct consequence of identifying all local
minima, we obtain global minimizers of $f(x)$ as well.
Identifying such minima is important 
in various modern applications,
such as parameter tuning of \Rev{stochastic} event simulators~\cite{diSerafino2010,mohan21,li2017hyperband,maclaurin2015gradient} \Rev{and
identifying optimal parameters within a quantum approximate optimization algorithm (QAOA)~\cite{farhi2014quantum,
		farhi2014quantumbounded}. }

In this paper, we propose a
multistart algorithm for nonconvex stochastic optimization
(MANSO) to identify all the local minima of $f(x)$ in $\cD$ using state-of-the-art
local stochastic optimization methods.  
(\defref{ILM} formally defines the meaning of  
``identifying'' a local minimum.) It is designed to extend
the popular multistart algorithm
Multi-Level Single Linkage (MLSL)
\cite{RinnooyKan1987,RinnooyKan1987ii}, which guarantees to find all the 
local minima of a nonconvex deterministic function.
MANSO 
uniformly samples points in $\cD$ and starts a local optimization method, 
such as stochastic gradient descent, from points that are
considered the best points in their neighborhood. 
By design, MANSO seeks to use as few  
local optimization  runs as possible by carefully
shortlisting sampled points based on the noisy observations of $f(x)$. In particular,
like MLSL, we construct a specific sequence of radii $\{r_k\}$, where $k$ is the 
number of iterations of MANSO, which enables us to certify that a uniformly 
sampled point is indeed probabilistically best to start a local stochastic
optimization run.

We also provide theoretical guarantees on the performance of MANSO. Primarily, under
assumptions that the stationary points of \Rev{the} true objective function $f$ in $\cD$ are 
sufficiently separated and the local stochastic optimization method is guaranteed 
to converge to the first-order stationary points with high probability, we prove 
that MANSO identifies all local minima of $f$ on $\cD$ with high probability. 
Furthermore, we show that MANSO does so while starting only finitely many 
local stochastic optimization runs. 

In addition, we demonstrate the performance of MANSO on two benchmark nonconvex
stochastic optimization problems (up to 10 dimensions) using a derivative-free 
adaptive sampling trust-region stochastic optimization algorithm (ASTRO-DF)~\cite{Shashaani2018} 
as a local stochastic optimization method.

\paragraph{\textbf{Related Work}} In general, a nonconvex optimization problem is NP-hard
even in the deterministic case, where the true function $f(x)$ can be evaluated
at any point $x\in \cD$. 
In the past few decades, extensive
research has been conducted to develop \Rev{stochastic} algorithms for solving nonconvex
deterministic optimization problems, with the aim of finding global optima.
Such \Rev{stochastic} methods can be broadly categorized into
multistart~\cite{RinnooyKan1987,RinnooyKan1987ii,Larson2018} and Bayesian
optimization(BO)~\cite{Kushner1964,Frazier2018} methods. Also, some recent
works leverage BO techniques for exploring the domain in the multistart
framework~\cite{Mathesen2020}.
In contrast,
the parallel work for stochastic nonconvex optimization still requires
much attention. 

In the deterministic case, one of the popular approaches is to 
use a stochastic search technique where
randomly sampled points in a compact domain are explored with the  help of a
local search method and a specific set of rules. These rules are designed to 
qualify any new
randomly sampled point to start a local search, and they depend on the already 
observed values of the points
in the neighborhood of the point being tested. 
With the help of such rules, these types of algorithms
try to avoid finding the same local minima multiple times. Such types of algorithms
are known as \textit{multistart} algorithms~\cite{RinnooyKan1987,RinnooyKan1987ii,Locatelli1998}.
Among various multistart algorithms, one of the most popular is
MLSL~\cite{RinnooyKan1987,RinnooyKan1987ii}, which provides
asymptotic guarantees in identifying  all the local minima assuming  all stationary
points of the objective function are separated by a positive distance.
Locatelli~\cite{Locatelli1998} improves upon this assumption of separated stationary
points and introduces N-MLSL (non-monotonic MLSL). Parallel
implementation of MLSL is also discussed
in~\cite{Larson2018} with asymptotic guarantees to find all local
minima when the true function $f(x)$ can be concurrently evaluated at 
any $x\in \cD$, unlike MLSL. 
\Rev{Mathesen et al.~\cite{Mathesen2020} propose a novel approach, where the restart
	points are decided based on ideas developed in BO literature~\cite{Frazier2018,Nguyen2017,Wessing2017}. Instead of
	randomly selecting the restart points from the search domain, their method
	chooses points based on a surrogate Gaussian process model of
	the objective function. 
  Naturally, their surrogate model is updated sequentially using the
	observed function values at points that are already evaluated. A detailed
	review of such
	algorithms, including significant recent developments, can be found
	in~\cite{Frazier2018}. 
	Similarly, the methods in~\cite{Krityakierne2015,Peri2012AMG,Regis2012} consider 
	other approaches for utilizing a surrogate model within a multistart framework.
	Multistart methods
  that seek improved efficiency have considered early termination of local
		searches~\cite{Zilinskas19} and tunneling and evolutionary
    strategies~\cite{Zheng21}. We note that these methods are for deterministic
    objectives. Extending such approaches to the stochastic setting may yield
    similar improvements. } 

 Our algorithm extends the MLSL algorithm to the case 
when we have access only to noisy function values. Like MLSL, we also rely on 
state-of-the-art local  stochastic optimization techniques
with first- and second-order convergence
guarantees~\cite{jin2018local,Ghadimi2014,Ghadimi2015,Shashaani2018}. However, to the best 
of our knowledge, no work has developed a multistart algorithm for stochastic
function evaluations. 

Here is a brief roadmap of the paper. In the next section, we define notations
and definitions used in the paper, followed in~\secref{assumptions}
by the assumptions required to prove
the theoretical properties of MANSO.
In~\secref{algorithm} we describe the details of our proposed method MANSO,
and in~\secref{analysis}
we establish its theoretical properties.
In~\secref{QAOA} we 
describe the
parameters of the quantum approximate optimization algorithm, which we will use for testing the MANSO algorithm.
In \secref{numerical} we present our numerical results.
We conclude in \secref{concl} with a summary and brief description of further work.

\section{Notations and Definitions}\label{sec:notation}

We use capital calligraphic letters to denote Lebesgue measurable subsets of $\R^d$.
The volume (Lebesgue measure) of a set $\mathcal{A}$ is denoted as
$m(\mathcal{A})$. We use $\left|A \right|$ to denote the cardinality of a
discrete set $A$. Unless otherwise stated, $\left\| \cdot \right\|$ denotes the
Euclidean norm. We denote the set of natural numbers as $\N$. 

We use $\Var[\xi]{\cdot}$ and $\Cov[\xi]{\cdot,\cdot}$ to represent the variance
and covariance, respectively, with respect to the random variable $\xi$. 
If $\{A_n(\xi)\}_{n \in \N}$ is a sequence of events and 
$P_{\xi}(\{A_n(\xi)\})$ is the probability of event $n$ occurring, then 
this sequence of events occurs with high probability (w.h.p.)
if for any $\delta>0$ there exists an $n_0\in 
\N$ such that $P_{\xi}(A_n(\xi))>1-\delta$ for any $n\geq n_0$. 
\Rev{We index the probability measure of the random variable $\xi$ by $\xi$
itself just to differentiate it from the uniform sampling measure used subsequently
in the paper. Moreover, the samples space of the random variable $\xi$, 
$\Omega$ is arbitrary and the range of $\xi$ is an arbitrary Borel measurable set.} 
When a sequence
$\{a_t\}$ is $O(b_t)$, it implies that there exist a $K>0$ and $t_0\geq 1$ such
that $\forall t\geq t_0$, $a_t\leq K b_t$. We let $\hat f_n(\cdot)$ 
denote an estimate of $f(x)$ constructed using $n$ \iid~measurements of $F(x,\xi)$.

We now establish notation to be used to describe the MANSO algorithm.
Let $k\in \N$ denote the number of iterations of MANSO. Let $S_k \subseteq \cD$
be the collection of uniformly sampled points up to and including iteration $k$. Let
$A_k \subseteq S_k$ be the collection of sampled points from which a local
stochastic optimization method has been started and is still active at 
iteration $k$, and let $L_k$ denote the set of points generated from all the local
stochastic optimization runs up to and including iteration $k$.
Let $X^* \subset \cD$
denote the set of local minima of $f$ in $\cD$ and $\hat X^*_{k}$
denote
the set of local minima identified before iteration $k$ of
MANSO. 
Let $Y^*$ denote the set of stationary points of $f$ on $\cD$ that are not local minima.
We let $\{\mathbf{X}_i^a, i\in \N\}$ denote the random
sequence of iterates generated by a given local stochastic optimization method
started from $a\in \cD$. $\{\mathbf{X}_i^a, i\in \N\}$ is a stochastic process
defined on the probability space $(\Omega,\mathbf{F}^a,P_{\xi})$. Let $\{x_i,
i\in \N\}$ be its realization.
We represent the
filtration (information) available at iteration $i$ of the local stochastic optimization method started from $a\in \cD$
as $\{\mathbf{F}^a_i\}$, which is an increasing family of $\sigma$-algebras of
$\{\mathbf{F}^a\}$ on which the stochastic process $\{\mathbf{X}^a_i\}$ is defined. We also call
$\{\mathbf{F}_i^a\}$  a \textit{filtration} at iteration $i$ of the local
stochastic optimization (LSO) run started from $a\in \cD$.

We now list other important notation:
\begin{itemize}
  \item Let $r_k = \frac{1}{\sqrt{\pi}} \sqrt[d] {\fGamma(1+{d/2})m(\cD) \sigma
    \frac{\log|S_k|}{|S_k|}}$ for some fixed $\sigma>4$, where $\fGamma$ is the gamma function.
    The radius is used in~\lemref{RK2} to define the neighborhood of 
     a candidate ``probabilistically best'' point.
    \item Let $\partial \cD$ denote the boundary of the set $\cD$. 
    \item Let
    $\B{y}{r} := \{x \in \cD: \|x-y\|\leq r \}$ represent a ball of radius $r$
    centered at any $y\in \cD$.
    \item For $\tau > 0$, let 
    \begin{align}\label{eq:Q_def}
    \cQ_{\tau}:= \cup_{ y \in \partial \cD} \{ x \in \cD: \|x-y\| < \tau,  \}. 
    \end{align}
    be the points in $\cD$ within $\tau$ of the boundary.
    \item We denote $\eta$ as the minimum distance between any two distinct stationary points of $f$ on $\cD$,
    that is,
    \[
    \eta=\min_{\{x,y\}\in X^*\bigcup Y^*, x\neq y} \|x-y\|.
    \] 
\end{itemize}

Next, we define the \emph{domain of attraction for method $M$} for any local minima.
\begin{definition}[Domain of attraction] \label{def:doa}
    For any $x^* \in X^*$, the domain of attraction
    $\mathcal{L}_{x^*}$  is defined as a subset of $\cD$ such that if
    a local stochastic method $M$ is started from any point in it, then it will
    converge to the local minimum $x^*$ w.h.p. 
    Formally, for any $\e>0$,
    \begin{align}
        \mathcal{L}_{x^*}=  \bigcup \left\{{a\in \cD} : \lim_{k\to \infty}P_{\xi}(\|\mathbf{X}_k^a-x^*\|>\e) =0 \right\}.
    \end{align} 
\end{definition} 

Note that the domain of attraction is defined only for those points that
the local stochastic method $M$ will converge to with high probability.
Moreover, it is not required that the respective domain of attraction for each
$x^*\in X^*$ partition $\cD$. \Rev{We consider \defref{doa} to be a reasonable
stochastic extension of the domain of attraction for the deterministic case
considered in the original MLSL paper and its extensions (e.g., \cite[Theorem
4]{RinnooyKan1987}, \cite{Locatelli1998}, \cite[Assumption 2]{Larson2018}).} 

We also define the event \emph{$\omega-$identifying a local \Rev{minimum}}, which
we use throughout. In our algorithm, $\omega$ is a tuning parameter and \Rev{is} given as
input by the user. 
\begin{definition}[$\omega$-identifying local minima]\label{def:ILM}
  Let $\omega\leq \frac{\eta}{2}$ be fixed, and let $\{\mathbf{X}^a_i\}$ be the sequence of
  iterates generated by the local method $M$ started from the point $a \in \LL_{x^*}$.
  A local method $M$ has identified the local minimum $x^*$ when the
  event $\left\{ \|\mathbf{X}^a_i - x^*\|<\omega  \right\} $ occurs for all $i \geq i_0$.
\end{definition} 

\section{Assumptions}\label{sec:assumptions}
We now state the assumptions needed for our theoretical analysis.
We group them into
assumptions about the objective and domain of the problem \eqref{eq:prob_def},
assumptions about the true function $f$ and its estimate computed by using
$F(x,\xi)$ at $x\in \cD$,
and assumptions about the
local stochastic optimization method used within the multistart framework.

\begin{assumption}\label{assm:prob}
  We first impose the following conditions on the problem \eqref{eq:prob_def}.
  \begin{enumerate}
    \item $\cD$ is a compact set, and $f$ is twice continuously differentiable on $\cD$. \label{sub:c2}
    \item The minimum distance between any two distinct stationary points is
      positive, that is, $\eta>0$.
          \label{sub:min_spacing}
    \item There exists $\tau>0$ such that $(X^* \bigcup Y^*) \bigcap \cQ_{\tau}
      = \emptyset$, for $\cQ_\tau$ defined in \eqref{eq:Q_def}.
      \label{sub:min_boundary}  
  \end{enumerate}
\end{assumption} 
Because the set $\cD$ is compact by \subassref{prob}{c2},
\subassref{prob}{min_spacing} implies that $f$ has finitely many stationary
points in $\cD$ and that $f$ is not flat in $\mathcal{D}$. 
\subassref{prob}{min_boundary} ensures \Rev{that there are no stationary points near the
boundary of $\cD$. (The parts of \assref{prob} are the same as those considered in the deterministic case~\cite{RinnooyKan1987}.) It is useful to have notation for
the set of points in $\cD$ within $\omega$ of a stationary point of $f$:}
For any $\omega \in (0,\eta)$, 
let 
\begin{align}\label{eq:T_def}
\cT_{\omega}:= \cup_{x\in \{X^*\cup Y^* \}} \B{x}{\omega}
\end{align}

\begin{assumption}\label{assm:Fhat} 
  For any $x\in \cD$, we assume that the estimate $\hat f_n(x)$ of
  $f(x)$, which is
  constructed by using $n$ \iid~measurements
  of the measurable function $F(x,\xi)$, satisfies
  \begin{enumerate}
    \item $\E[\xi]{\hat f_n(x)}= f(x)$ 
    \item $\Var[\xi]{\hat f_n(x)} < \infty$.
    \end{enumerate}
\end{assumption} 
Since $\hat f_n(\cdot)$ is constructed by using $n$ \iid~measurements
of the measurable 
function $F(\cdot,\xi)$, this assumption requires $F(\cdot,\xi)$ to satisfy
some regularity conditions. 
For instance, $\hat f_n(\cdot) =
\frac{1}{n}\sum_{i=1}^{n}F(\cdot,\xi_i)$ is an unbiased estimate of $f$ and $\Var[\xi]{\hat f_n(x)}<\infty$ if $\Var[\xi]{F(x,\xi_i)}<\infty$.

\begin{assumption}\label{assm:LSO} 
  We make the following assumptions about the local stochastic optimization method $M$.
  If $\{\mathbf{X}^a_i\}$ is a random sequence of iterates produced by $M$
  when started from $a \in \mathcal{L}_{x^*}$, then it 
  satisfies the following.
\begin{enumerate}
\item For any $x^* \in X^*$, $m(\mathcal{L}_{x^*})>0$ and $m(\mathcal{L}_{x^*}
  \setminus \cQ_{\tau}) > 0$  \label{sub:measures}
\item For any two local minima $\{x^*,y^*\}\in X^*$,  
    $\mathcal{L}_{x^*} \cap \mathcal{L}_{y^*} = 
    \emptyset$.\label{sub:overlap}
  \item  For any $\nu >0$ there exist an $i_0 \in \mathbb{N}$, $\omega \in
    (0,\frac{\eta}{2})$, and a sequence $\{\Lambda_i\}$ satisfying  $\Lambda_i
    \in (0,1)$ and $\underset{i \to \infty}{\lim} \Lambda_i = 0$ such that \label{sub:LSO_nu}
 \begin{enumerate}[label=(\roman*)]
   \item $P_{\xi}\{ \|\nabla f(\mathbf{X}_i^a)\|^2<\nu | a  \in  \mathcal{L}_{x^*} \} \geq 1 - \Lambda_i$  and  
   \item $P_{\xi}\left\{ \|\mathbf{X}_i^a - x^*\|<\omega ~\big| \|\nabla f(\mathbf{X}_i^a)\|^2<\nu, a  \in  \mathcal{L}_{x^*} \right\} = 1,$
 \end{enumerate}
    for all $i \geq i_0$.
\item For $x, y \in X^*$ and $x\neq y$, we also assume that the sequence of iterates 
  $\{\mathbf{X}^a_i\}$ and $\{\mathbf{Y}^b_j\}$ generated by LSOs started at $a \in
  \mathcal{L}_{x} $ and $b \in \mathcal{L}_{y}$, respectively.  Then \label{sub:LSO_omega}
\[ P_{\xi}\left\{  \min_{\forall \{i \geq 1\}} \|\mathbf{Y}^b_j- \mathbf{X}^a_i \|>\omega \Big| 
 \mathbf{F}^b_{j-1} , \mathbf{F}^a_{\infty} \right\}=1 \text{ for all } j \in \N .\]
\end{enumerate} 
\end{assumption} 

Some of the conditions on the local method $M$ in~\assref{LSO} are strong
conditions that may be difficult to satisfy in practice by a local stochastic
optimization method on a nonconvex problem.  \Rev{Yet, we find these
assumptions to be a natural stochastic version of their deterministic
counterparts: for Assumption 3.3 above is similar to the strictly decent property assumed for the (deterministic) local optimization method~\cite{RinnooyKan1987}.}
\subassref{LSO}{measures} ensures that the domain of attraction for any local minima 
has positive measure and that $\mathcal{L}_{x^*}$ does not
lie entirely in the boundary set $\cQ_{\tau}$. 
\subassref{LSO}{overlap} ensures that no two local minima have overlapping
domains of attraction. 
In general, local
stochastic
methods \cite{Ghadimi2014,Ghadimi2015} guarantee convergence to a
stationary point only w.h.p.; that is, they satisfy \subassref{LSO}{LSO_nu}(i).
In addition, we need \subassref{LSO}{LSO_nu}.1(ii) to ensure that
$x^*$ is within an $\omega$-ball of all iterates after a large number of iterations with
probability 1, given that the local method is started in that $\mathcal{L}_{x^*}$ and
the norm of the gradient at the last iterate (of an LSO run
) is small enough (less than $\nu$).
Furthermore, in \subassref{LSO}{LSO_omega},
we assume that the iterates (realizations of \Rev{the} sequence of iterates) generated by
any two LSO runs started in different domains of attractions are at least
$\omega$ apart. These conditions are necessary for developing 
the algorithm and for showing that the algorithm MANSO identifies all the local
minima
w.h.p.~using the LSO method $M$.

\section{Statement of the Algorithm}~\label{sec:algorithm}
\tabref{CON} lists the conditions that our algorithm checks when deciding where
to start an LSO. \algref{aposmm-s} states our algorithm for
$\omega$-identifying all local minima of $f(x)$.

\begin{table}
    \centering
    \caption{Conditions defined by tolerances $r_k$, $\tau>0$, $\omega>0$,
    $n\in \mathbb{N}$, and $\beta\in(0,1/2)$ to be checked by~\algref{aposmm-s}
    before starting a local optimization run. \label{tab:CON}}
    \begin{tabular}{@{}l@{}}
        \textbf{S1} $\nexists$ a point $z \in \B{a}{r_k} \cap (S_k)$,
        such that 
        $ P_{\xi} \left(\hat f_{n}(z)- \hat f_{n}(a) > \sqrt{\frac{\Var[\xi]
                {\hat f_{n}(z)-\hat f_{n}(a) }}{\beta}} \right) \leq \beta$,\\
        \hspace{2em} where $n$ is the  
        number of random samples
        of $f(\cdot)$ at respective points.
        \\
        \textbf{S2} $a \notin \cup_{x\in \hat X_{k-1}^*} \B{x}{\omega}$, for a given
        $\omega < \frac{\eta}{2}$, where $\hat X^*_{k-1}$
        is the collection of 
         \\
         \hspace{2em}  approximate local minima up to iteration $(k-1)$. 
         \\
        \textbf{S3} $a \notin \cQ_{\tau}$, that is, near the boundary of set
         $\cD$.   
         \\
        \textbf{S4} $a$ has not started any LSO.
         \\
        \bottomrule
\end{tabular}
\end{table}

\begin{algorithm}[H]
  \textbf{Input:} LSO method $M$; $\beta \in (0,\frac{1}{2}) $; $\tau>0$; $\omega>0$; 
  sampling effort $n\in \N$.  \\
  
  Initialize $S_0 = A_0 = \hat X_0^*= L_0 = \{\},$ 
  
  \For{$k=0,1,\ldots$}{
  Uniformly sample a point  $a$ in $\cD$, evaluate $\hat{f}_{n}(a)$ 
  and add it to $S_k$.

    Start LSO method $M$ from all points in $S_k$ satisfying the conditions listed
     in~\tabref{CON} for a given $r_k$, $\beta$, $\sigma$, $\omega$, and sampling 
     effort $n, \forall x \in S_k$. Add those points to $A_k$.
   
    Update $L_k$ by adding the next iterate generated by each LSO started from 
    all the points in $A_k$. \label{al:1step}
  
    Terminate any LSO run if its current iterate
        is within $2\omega$ distance of any point in $L_{k-1}$ (from other LSOs) 
        and remove that LSO run from $A_k$.\label{al:term}
  
   Update $\hat{X}_k^*$ by adding any new local minima identified during the 
   local searches, and remove that LSO run from $A_k$.
  
  Set $S_{k+1} = S_k$, $ \hat{X}_{k+1}^* = \hat{X}_k^* $, $L_{k+1}= L_k$, and
  $A_{k+1}= A_k$.
}
  \caption{MANSO\label{alg:aposmm-s}} 
\end{algorithm}
Notice that in condition \textbf{S1} of~\tabref{CON}, we have used $n \in \N$ as the number of random 
samples of $f(\cdot)$ at respective points. We
show in \thmref{finite} later that for any $n \in \N$ the total number of LSO
runs started by~\algref{aposmm-s} is finite.  Since the results on the
diffusion approximation of nonconvex SGD~\cite{Hu2019} show that the use of
smaller batch sizes in batch SGD methods help escape sharp local minima
and nondegenerate saddle points,  we anticipate that the use of a larger $n$
would be ``better and faster'' in $\omega$---identifying all the local minima. 

Note that MANSO can be viewed as a stochastic analogue of MLSL. That is, if we
further assume that (1) $\hat f_n(x) \to f(x)~ P_{\xi}-a.s.$ as $n \to
\infty$ and (2) $\Var[\xi]{\hat f_n(x)}\to 0$ as $n \to \infty$ , and for any two
points $\{x,y\}\in \cD $, $\Cov[\xi]{\hat f_n(x),\hat f_n(y)} \to 0$ as $n \to
\infty$, then the above notion (see~\tabref{CON}, \textbf{S1}) of not finding a
``probabilistically best'' point $z$ in its $r_k-$neighborhood converges to the
original MLSL~\cite{RinnooyKan1987,RinnooyKan1987ii} condition of finding a
``better'' point $z$, as $n \to \infty$. Also, while implementing MANSO we
estimate $\Var[\xi] {\hat f_{n}(x)-\hat f_{n}(y) }$ using samples of
$f(\cdot)$ for any $\{x,y\}\in \cD$.  In step~\ref{al:term}
of~\algref{aposmm-s}, since we do not have prior knowledge about the minimum
separation between stationary points, $\eta$, we  must choose a small enough
positive  value for each $\omega$ and $\tau$.

\section{MANSO Asymptotic Analysis}~\label{sec:analysis}

MANSO seeks to use as few 
LSO runs  
as possible to $\omega$-identify all 
local minima of $f$ in $\cD$. 
We will show that the total number of LSO runs started by MANSO is finite,
even if the algorithm is run forever. 
\Rev{We state our main theoretical results in this section, but to ease presentation, 
the proofs of lemmas are deferred to the appendix.

We first show a limiting result for the measure of balls} around any point $a \in \cD \setminus 
(\cQ_{\tau}\cup \cT_{\omega})$,
where $\cQ_\tau$ and $\cT_{\omega}$ are defined in \eqref{eq:Q_def} and
\eqref{eq:T_def}, respectively. 

\begin{lemma}\label{lem:RK} 
  Under~\subassref{prob}{c2} and~\ref{assm:Fhat}, for any $a \in \cD \setminus 
  (\cQ_{\tau}\cup \cT_{\omega}) $, $\beta\in (0,1/2)$, and
  for all $n\geq 1$
   \[ \lim_{r
  \to 0} \frac{m(\A{a}{r}{n}{\beta})}{m ( \B{a}{r} )} \geq \frac{1}{2}, \]
where $$\A{a}{r}{n}{\beta} := \left\{x\in \cD: \| x-
a\|\leq r \text{ and } P_{\xi} \left(\hat f_n(x) - \hat f_n(a) > \e_n(x;a) \right)
\leq \beta \right\},$$ and $\e_n(x;a)=\sqrt{\frac{\Var[\xi]{\hat f_n(x)
            -\hat f_n(a) }}{\beta}}$. 
\end{lemma}

Specifically, \lemref{RK} derives a bound on the measure of the set of points that are
not probabilistically best in a ball of radius $r$ around a point $a$ drawn
uniformly from the set of points in $\cD$ not within $\omega$ of a stationary
point or $\tau$ of the boundary of $\cD$. (Not being probabilistically best is
measured with respect to $\beta$, an integer $n$, and the tolerance
$\e_n(x;a)$.) \lemref{RK} shows in the limit as $r$ converges to zero that the set of 
not probabilistically best points has \Rev{a} measure of at least half of the ball
around $a$. Note that it is true for any $n\in \mathbb{N}$.

Next, we use \lemref{RK} and  construct a specific sequence of
radius $\{r_k\}$ to show that
 the probability of starting an LSO run from any previously sampled point is bounded by a 
 term that converges to zero as the number of iterations increases. The proof of the following
  result uses arguments similar to those used in \cite[Theorem 8]{RinnooyKan1987}.   
\begin{lemma}\label{lem:RK2}
    Let $t_k'$ be the number of LSO runs started from any point in $ S_k\setminus
    (\cQ_{\tau}\cup \cT_{\omega})$ during iteration $k$ of \algref{aposmm-s}. 
    Then under~\subassref{prob}{c2} and ~\ref{assm:Fhat}
    and for $r_k$ as defined in \secref{notation} and used in  \tabref{CON} with $\sigma > 0$, 
    \[ P[\{ t_k' > 0\}] = O(|S_k|^{1- \frac{\sigma}{2} }).\]
    \end{lemma}

Subsequently, we use the summability of the bound obtained in~\lemref{RK2}
on the probability of starting any LSO run from the set of sampled points
not within $\omega$ of any stationary point or $\tau$ of the boundary of $\cD$ to show
in~\thmref{finite} that the total number of LSO runs started
by~\algref{aposmm-s} is finite. Furthermore, condition \textbf{S3} in~\tabref{CON}
ensures that LSO does not start from any point in $\cQ_{\tau}$, 
and step~\ref{al:term}~in~\algref{aposmm-s} 
ensures that the total number of LSO runs started from
points sampled within the $\omega$-ball of any stationary point is finite  as
$k \to \infty$. (The total number of stationary points is finite because of
~\subassref{prob}{min_spacing}.)

That is, the number of LSO runs started by MANSO is finite even if MANSO runs forever.
\begin{theorem}\label{thm:finite}
  Let $t_k$ be the number of LSO runs started by MANSO in iteration $k$. Then under~\assref{prob},
   $\sum_{k=1}^{\infty}t_k < \infty$ with probability 1. 
\end{theorem}

\begin{proof}
	For any $\sigma>4$ and $|S_k|= O(k)$, \lemref{RK2} implies that
	\begin{align} \label{eq:fin0}
		\sum_{k=1}^{\infty} P[ \{t_k' > 0\} ] < \infty,
	\end{align}
	where $t_k'$ is the number of LSO runs started from points in $a\in S_k\setminus
	(\cQ_{\tau}\cup \cT_{\omega})$ during iteration $k$ of \algref{aposmm-s}. 
	Note that using the first Borel--Cantelli lemma \cite[Theorem 2.3.1]{durrett2010probability}, 
	we have from equation~\eqref{eq:fin0} that
	\[
	P\left(\cap_{i\geq 1} \cup_{k\geq i } \{ t_k' >0 \}\right) = 0. 
	\] 
	This is equivalent to 
	\begin{align}
		P\left( \exists i\geq 1: \forall \ k \geq i, \{t_k' \leq 0 \}\right) = 1.
		\label{eq:BC2}
	\end{align}
	Since $t_k'\geq 0, \forall k\geq 1$, the result in~\eqref{eq:BC2} implies
	that $\lim_{k \to \infty} t_k' \to 0~ P-$ almost surely, that is, with
	P-probability 1. Since $\{t_k'\}$ is a sequence of natural numbers, it
	implies that 
	\begin{align}
		\sum_{k=1}^{\infty}t_k' < \infty. 
		\label{eq:BC3}
	\end{align}
	
	Furthermore, if a point in $S_k$ belongs to $\cQ_{\tau}$, then MANSO (see \textbf{S3}
	in \tabref{CON}) does not start a run at that point, since we assumed
	in~\subassref{prob}{min_boundary} that no minimum lies in $\cQ_{\tau}$. 
	Therefore, the probability of
	starting an LSO run from any point $a\in S_k \cap \cQ_{\tau}$ is zero, and the
	result in~\eqref{eq:BC3} holds true  for all $a \in ( S_k \setminus
	(\cQ_{\tau} \cup \cT_{\omega}) )\cup \cQ_{\tau}  $, that is, for all $a\in S_k
	\setminus \cT_{\omega}$.
	
	Now consider the last case when $a\in S_k \cap \cT_{\omega}$. Let $t_k''$ be
	the number of LSO runs started by MANSO in the iteration $k$ from any
	point $a\in S_k \cap \cT_{\omega}$. Since in step~\ref{al:term}~of~\algref{aposmm-s} we 
	kill the LSO run if its current iterate is within $2\omega$ distance of an 
	already generated iterate from any of the previous LSO, 
	a run can be started from any $a\in S_k \cap \cT_{\omega}$ at most once. Since 
	\subassref{prob}{min_spacing} implies that there are only a finite number of 
	local minima, there exists a $k_0 \geq 1$ such  that for all $k\geq k_0$ 
	the number of LSO runs started
	at points $a\in S_k \cap \cT_{\omega}$ will be zero. 
	Therefore
	the number of LSO runs $t_k'' \to 0$ as the
	number of samples increases to infinity. 
	Since $\{t_k''\}$ is a sequence of
	natural numbers, it implies that 
	\begin{align}
		\sum_{k=1}^{\infty}t_k'' < \infty. 
		\label{eq:BC4}
	\end{align}
	Since $t_k= t_k' + t_k''$, the result follows immediately by adding~\eqref{eq:BC3} 
	and~\eqref{eq:BC4}.
	\flushright \qed
\end{proof}

Our next goal is to show that under certain assumptions MANSO will
$\omega$-identify all the local minima w.h.p. Recall that $X^*$ is the collection of
local minima of $f$ in $\cD$. 
The~\subassref{LSO}{LSO_nu}  guarantees that if MANSO starts an LSO run from any point in the
domain of attraction $\mathcal{L}_{x^*}$ of a local minimum $x^* \in X^*$, then the 
LSO method $M$ identifies it w.h.p. 
In particular, we prove that if a point is sampled in the
domain of attraction of a local minimum, then the  probability of  that local
minimum not being identified is sufficiently small for large enough iterations 
of both MANSO and LSO. 
Combined with the fact that these domains of attraction are of
positive measure and the probability of getting a uniformly sampled point in
any domain of attraction approaches 1 as the number of iteration increases
(since the number of sampled points increase with each iteration), we show
that all the local minima are identified w.h.p.
\begin{theorem}\label{thm:ID}
    Under~\assref{LSO}, and given the sequence of radii $\{r_k\}$ as constructed
    in~\lemref{RK2} and used in condition \textbf{S1} of \tabref{CON}, MANSO identifies all
    the local minima of \eqref{eq:prob_def} w.h.p.
\end{theorem}

\begin{proof}
	Let $a$ be a point sampled in iteration $k$ of~\algref{aposmm-s} in the domain of attraction
	$\mathcal{L}_{x^*}$ of the local minimum $x^*$ (see Definition~\ref{def:doa}) 
	for the first time and none of the points sampled before belong to $\mathcal{L}_{x^*}$. 
	Since
	$m(\mathcal{L}_{x^*})>0$ because of~\subassref{LSO}{measures},  the probability of obtaining at least a 
	uniformly sampled point in $\mathcal{L}_{x^*}$ approaches 1 as the number of
	samples increases to infinity \cite{Brooks1958}. We also assume that  the 
	local minimum $x^*$ has not been identified yet.
	
	Recall $\hat X^*_{k}$ is the collection of approximate local minima
	identified up to iteration $k$ and that set $A_k$ is the collection of
	sampled points from which the LSO run has started and is still active up to 
	iteration $k$. Recall the definition (see Definition~\ref{def:ILM}) 
	of the event that the local
	minimum is identified at any arbitrary iteration $l$ as \[ I_{l}=: \{
	\forall~l' > l~\exists p \in A_{l'}~: \|\mathbf{X}^{p}_{ l'}-x^*\|<\omega\}, \] where
	$\{x^{p}_{l'}, l\in \N\}$ is the sequence of iterates generated by LSO
	at $p$.  Now, let us  compute the  probability that the local minimum has been
	identified up to the iteration $\bar k \geq k$.
	Observe that
	\begin{align}
		\nonumber
		&P_{\xi}(I_{\bar  k}|S_{\bar k} \cap \mathcal{L}_{x^*} \neq \emptyset)
		\\
		\nonumber
		&= 1- P_{\xi}(I_{\bar  k}^C|S_{\bar k} \cap \mathcal{L}_{x^*} \neq \emptyset)
		\\
		\nonumber
		&= 1- \bigg(P_{\xi}(\text{LSO has started from $a$ in iteration $k$},I_{\bar  k}^C 
		| S_{\bar k} \cap \mathcal{L}_{x^*} \neq \emptyset ) 
		\\
		&\quad + P_{\xi}(\text{ LSO has not started from $a$ in iteration $k$ },I_{\bar  k}^C 
		| S_{\bar k} \cap \mathcal{L}_{x^*} \neq \emptyset) 
		\bigg). 
		\label{eq:t1}
	\end{align}
	We first analyze the first probabilistic term  in~\eqref{eq:t1}. Since LSO starts at
	$a$, 
	then by definition of $A_k$, $a \in A_{k}$. Even if we sample another
	point in $\mathcal{L}_{x^*}$ at any iteration $\bar k>k$, we will have
	$a\in A_{\bar k}$ because we terminate only the latest LSO in
	step~\ref{al:term} of MANSO. Also, if at some iteration $\bar k>k$ any of
	the iterates generated by some other sampled point, $b$ (at iteration
	$k'<k$) in $\mathcal{L}_{y^*}$, for $x^*\neq y^*$, jumps into
	$\mathcal{L}_{x^*}$, then because of~\subassref{LSO}{overlap}
	and~\subassref{LSO}{LSO_omega} we will still have $a \in A_{\bar k}$ on
	the event $I_{\bar  k}^C$, since LSO at $a$ will not be terminated at
	step~\ref{al:term} of MANSO because
	\[ P_{\xi}\left\{   \|\mathbf{Y}^b_{\bar k-k'}- \mathbf{X}^a_{\bar k-k} \|>\omega \Big| 
	\mathbf{F}^b_{\bar k-k'-1}   \right\}\geq P_{\xi}\left\{  \min_{\forall \{i \geq 1\}} \|\mathbf{Y}^b_{\bar k-k'}- \mathbf{X}^a_{i} \|>\omega \Big| 
	\mathbf{F}^b_{\bar k-k'-1} , \mathbf{F}^a_{\infty} \right\}=1 \text{ for all } \bar k -k \in \N .\]
	Consequently, since $a \in \mathcal{L}_{x^*}$ and it remains in $A_{\bar k}$ for any $\bar k > k$, then because
	of~\subassref{LSO}{LSO_nu} there exists a $k_0 \in \N$ such that for any iteration $\bar
	k -k > k_0$ of LSO at $a$, 
	we have
	\begin{align}
		P_{\xi}&(\text{LSO has started from $a$ in iteration $k$, }  I_{\bar  k}^C | S_{\bar k} \cap 
		\mathcal{L}_{x^*} \neq \emptyset ) 
		< \frac{1}{2}\Lambda_{\bar k- k}.
		\label{eq:t2}
	\end{align}
	Note that we use the same iteration counter for the LSO run and MANSO, since we
	progress one step of each  active LSO in $A_{\bar k}$ in each  iteration of
	MANSO (see Step~\ref{al:1step}). 
	Choosing $\bar k$ large enough such that $\bar k-k > k_0$, we obtain the
	last inequality in~\eqref{eq:t2} by using~\subassref{LSO}{LSO_nu}, since
	LSO at $a$ never gets terminated given that the local minimum $x^*$ has not been
	$\omega$-identified.
	
	Next, we analyze the second probabilistic term in~\eqref{eq:t1}. In this 
	case LSO may not start from $a$ in iteration $k$, since it gets
	rejected because of any of the following conditions
	from~\tabref{CON} not being satisfied by $a$ in the iteration $k$.
	Using these conditions in the third term of~\eqref{eq:t1}, we have
	\begin{align}
		\nonumber
		P_{\xi}(&\text{ LSO has not started from $a$ in iteration $k$},I_{\bar  k}^C 
		| S_{\bar k} \cap \mathcal{L}_{x^*} \neq \emptyset )
		\\
		\nonumber
		&\leq P_{\xi}(\text{S(1)},I_{\bar  k}^C 
		| S_{\bar k} \cap \mathcal{L}_{x^*} \neq \emptyset ) +  
		P_{\xi}(\text{S(2)},I_{\bar  k}^C 
		| S_{\bar k} \cap \mathcal{L}_{x^*} \neq \emptyset )
		\\
		\nonumber
		& \quad +P_{\xi}(\text{S(3)},I_{\bar  k}^C 
		| S_{\bar k} \cap \mathcal{L}_{x^*} \neq \emptyset ) +  
		P_{\xi}(\text{S(4)},I_{\bar  k}^C 
		| S_{\bar k} \cap \mathcal{L}_{x^*} \neq \emptyset ),
	\end{align}
	where $S(1)$ denotes the event that $a$ does not satisfy condition \textbf{S1} from
	~\tabref{CON} and similarly for \textbf{S2}, \textbf{S3}, and \textbf{S4}.
	The last three cases are straightforward to analyze. First consider event \textbf{S2}. Since we
	assumed that the local minimum $x^*$ has not been identified yet, then the
	event $\{a \in \cup_{x\in \hat X_{ k -1}^*} \B{x}{\omega}\}$ is of
	probability measure zero, and hence
	\begin{align} 
		P_{\xi}(\text{S(2)},I_{\bar  k}^C 
		| S_{\bar k} \cap \mathcal{L}_{x^*} \neq \emptyset) = 0.
		\label{eq:t3}
	\end{align}
	For event \textbf{S3}, $a\in \cQ_{\tau} $,  and we assumed at the beginning that $a \in
	\mathcal{L}_{x^*} $ as well. But because of~\subassref{prob}{min_boundary}
	$\mathcal{L}_{x^*} \cap \cQ_{\tau} = \emptyset$. Therefore it follows that
	\begin{align}
		P_{\xi}(\text{S(3)},I_{\bar  k}^C 
		| S_{\bar k} \cap \mathcal{L}_{x^*} \neq \emptyset ) =0 .
		\label{eq:t4}
	\end{align}
	Since we assumed at the beginning that for any $\bar k < k$ none of the
	uniformly sampled points belong to  $\mathcal{L}_{x^*}$,
	the final event 
	\textbf{S4} is an impossible event, and therefore 
	\begin{align}
		P_{\xi}(\text{S(4)},I_{\bar  k}^C 
		| S_{\bar k} \cap \mathcal{L}_{x^*} \neq \emptyset ) = 0.
		\label{eq:t5}
	\end{align}
	For the first event \textbf{S1} we assumed that there exists $\bar a \in
	\B{a}{r_{ k}} \cap (S_{ k})$ that does not satisfies \textbf{S1} for a point  
	$a$ in $\mathcal{L}_{x^*}$. Now notice that
	\begin{align}
		\nonumber
		P_{\xi}(&\text{S(1)},I_{\bar  k}^C 
		| S_{\bar k} \cap \mathcal{L}_{x^*} \neq \emptyset ) 
		\\
		\nonumber
		&= P_{\xi}(\bar a \notin \mathcal{L}_{x^*} ,I_{\bar  k}^C 
		| S_{\bar k} \cap \mathcal{L}_{x^*} \neq \emptyset )
		\\
		&\quad +P_{\xi}(\bar a \in \mathcal{L}_{x^*},I_{\bar  k}^C 
		| S_{\bar k} \cap \mathcal{L}_{x^*} \neq \emptyset ).
	\end{align}
	
	First consider the case when $\bar a \notin  \mathcal{L}_{x^*} $. Since $r_k
	\to 0$  as $k \to \infty$ and $a \in  \mathcal{L}_{x^*}$, there must exist a
	$k_0 \geq k $ such  that $\bar a \notin \B{a}{r_{\bar k}} \cap (S_{\bar
		k})$ for all $\bar k \geq k_0$. Therefore, LSO will start from $a$ at iteration
	$k_0$. Using arguments similar to those used in~\eqref{eq:t2},  
	there exists a $k'' \geq  k $ such that  $\forall \bar k \geq k''$
	\begin{align} 
		P_{\xi}(\bar a \notin \mathcal{L}_{x^*} ,I_{\bar  k}^C 
		| S_{\bar k} \cap \mathcal{L}_{x^*} \neq \emptyset ) \leq
		\frac{1}{2}\Lambda_{\bar k-k_0}.
		\label{eq:t6}
	\end{align}
	On the other hand,   because of our assumption that for any $\bar k < k$ none of the
	uniformly sampled points belong to  
	$\mathcal{L}_{x^*}$, the case $\bar a \in 
	\mathcal{L}_{x^*}$  is an impossible event, 
	and thus
	\begin{align}
		&P_{\xi}(\bar a \in \mathcal{L}_{x^*},I_{\bar  k}^C 
		| S_{\bar k} \cap \mathcal{L}_{x^*} \neq \emptyset )
		=0.
		\label{eq:t7}
	\end{align}
	Since $m(\mathcal{L}_{x^*})>0$ due to~\subassref{LSO}{measures},  the probability of
	obtaining at least a uniformly sampled point in $\mathcal{L}_{x^*}$ approaches 1 as
	the number of samples increases to infinity \cite{Brooks1958}. Therefore
	$\lim_{\bar k \to \infty} P_{\xi}( S_{\bar k} \cap \mathcal{L}_{x^*} \neq \emptyset) 
	= 1$. 
	Now, substituting equation~\eqref{eq:t2}-\eqref{eq:t7} into~\eqref{eq:t1},
	we obtain, for large enough $\bar k$,
	\begin{align}
		P_{\xi}(I_{\bar  k} |S_{\bar k} \cap \mathcal{L}_{x^*} \neq \emptyset ) > 1- \bar\Lambda_{\bar k},
	\end{align}
	where $\bar\Lambda_{\bar k} = \min \{\Lambda_{\bar k-k_0}, \Lambda_{\bar k-k} \}$. 
	Therefore, we have shown that any local minimum $x^*  \in  X^*$ will be
	identified w.h.p. Also, note that because of~\subassref{LSO}{LSO_nu}, 
	$\bar  \Lambda_{\bar k}$ decreases to 0 as the number of iterations $\bar k$ of
	LSO increases with the number of iterations of MANSO.
	\flushright \qed
\end{proof}

\section{Numerical Experiments}

We compare implementations of MANSO in their ability to 
solve difficult synthetic benchmark
problems and to identify optimal variational parameters within the quantum approximate optimization
algorithm (QAOA)~\cite{farhi2014quantum}. Our MANSO implementation
and scripts to perform our numerical experiments are available: 
\begin{center}\url{https://github.com/prat212/MANSO.git}\end{center}

\subsection{Synthetic Benchmark Experiments} \label{sec:numerical}

We benchmark our implementation of MANSO on 
nonconvex optimization 
problems with large variance in their observations. In particular we fix 
two non-convex benchmark functions, Branin-Hoo ($d=2$) and Shekel
 ($d=4,6,8 \text{ and } 10$), and make each non-convex objective evaluation
  stochastic by adding a Gaussian noise with variance 1. 
We add this significant noise only to make the testing of MANSO rigorous and robust.
We generate 10 sample paths of Gaussian noise to
create a set of 10 
problems each for the Branin and Shekel functions. We seek to find all
of the local minima for each problem within a fixed budget of function
evaluations, $B$. We use ASTRO-DF~\cite{Shashaani2018} as the local method; it
is a derivative-free
trust-region stochastic
 optimizer selected because of its theoretical guarantee to converge to
first-order critical points of the objective function. Other optimizers with
 convergence guarantees to first-order critical points (e.g.,~\cite{Ghadimi2015,Ghadimi2014})
  could naturally be used as the local method in MANSO.  
To improve performance in our numerical experiments,
we ensure that 
there are at most $10$ active LSO runs at any given iteration. That is at
each iteration $k$ of MANSO, we sample a point uniformly in $\cD$ if the
total number of active runs is no larger than a fixed threshold, heuristically 
set to 10 in the experiments. Naturally, MANSO ensures that the conditions
 listed in~\tabref{CON} are satisfied before starting an LSO run from all the sampled points.
 Note that for condition (S1) in~\tabref{CON}, for any two points $\{x,a\}\in
\cD$, our implementation estimates
$\Var[\xi]{\hat f_{n}(x) -\hat f_{n}(a)) }$ using the $n$ samples of $f(x)$ and $f(a)$.


We compare the
performance of MANSO with drawing points uniformly in the search domain. 
We compare only with such a random search method because 
we are unaware of other methods \Rev{that} aim to find all local optima of stochastic
nonconvex functions. 
Our empirical results demonstrate that MANSO
outperforms random search in identifying points that are within a small ball of all
local minima. As we would like random search to perform equally well across
problems independent of problem dimension, we consider each ball around a local
minimum to always have a volume that is a small fraction (e.g., $\nicefrac{1}{1000}$) of the volume of
the domain. This number is arbitrarily chosen but gives a sense that how fast MANSO
and random search can evaluate points near the local minima.
We measure the performance of methods using data profiles~\cite{JJMSMW09}.

\subsection{Data profiles}
Data profiles present the fraction of problems ``solved'' from a set of
problems $\mathcal{P}$ after a certain number of function evaluations 
by an implementation of method $h$ in a set of methods $\mathcal{H}$. The set of
implementations $\cH$ is created by adjusting $\beta, \omega,$ and $n$ of MANSO.
For a given objective function,
we create different problems by  changing the initial random seed. For our comparisons, the set 
$\mathcal{H}$ contains a uniform random sampling method and different versions of MANSO, obtained
by varying its tuning parameters $\omega, \tau, \beta$, and $n$. 
Data profiles mark a
problem instance $p\in \cP$ as solved based on a user-defined test
criterion. We use the 
criterion proposed in~\cite{Larson2018} to classify that a problem $p\in \cP$ is
solved: a problem is solved when a point is evaluated near each local minima for
the problem.
Let there be $j$ local minima for a given problem. 
Mathematically, we define a test that ensures that a local minima $x^*\in\{x_1^*,x_2^*,\ldots,x_j^*\}$ is  identified at level
$\rho_d(\zeta)$ after $e$ evaluations, as
\begin{align}
   \exists x \in \cE_e \text{ with } 
    \|x-x^*\|\leq \rho_d(\zeta)  ,
    \label{eq:tc}
\end{align} 
where $\rho_d(\zeta) = \frac{1}{\sqrt{\pi}} \sqrt[d] {\Gamma(1+{d/2})m(\cD)
\zeta}$ and the set $\cE_e$ is constructed by sequentially adding points the
number of times they are being evaluated by a method $h\in \cH$. That is, all
the points in $S_k \cup L_k$  till $e$ out of $B$ budget is used. Notice that
    the volume of a ball of radius
     $\rho_d(\zeta)$ is $\zeta$ times the volume of the search  domain $\cD$. 
Hereafter, we use~\eqref{eq:tc} for a problem
$p\in \cP$, method $h\in \cH$  and
for all $x^*\in \{x^*_i\}_{i=1}^{j}$ to compute
\[t_{p,h}(x^*)=\min \left\{e\geq 1: \exists x \in \cE_e \text{ with } 
\|x-x^*\|\leq \rho_d(\zeta)\right\}, \] 
and define the data profile metric for $e>0$ function evaluations as
\begin{align}
    d_h(e,x^*)= \frac{|\{p\in \cP: t_{p,h}(x^*)\leq e, \}|}{|\cP|}.
\end{align}

In the next section, we present the details of two benchmark nonconvex
functions on which we evaluated the performance of MANSO.

\subsection{Benchmark problems}
We consider the Branin--Hoo and Shekel (4, 6, 8,
and 10 dimensions) functions. Results for $d=6$ and $d=8$ Shekel problems appear in Appendix~\ref{app:1}. 

\paragraph{Branin--Hoo function:}
The Branin-Hoo 
function~\cite{branin} is a two-dimensional nonconvex problem with three local minima with the same
optimal value. Mathematically, it is defined as
\begin{align}
    f(x^{(1)},x^{(2)})= a(x^{(2)}-b(x^{(1)})^2+cx^{(1)}-r)^2+s(1-t)\cos(x^{(1)})+s,
\end{align}
where $a=1,b=5.1/(4\pi^2),c=5/\pi,r=6,s=10$ and $t=1/(8\pi)$ and $\cD=[-5,10]\times[0,15]$.

\paragraph{Shekel function:}
The Shekel function is the $d$-dimensional nonconvex problem with $m$ local
minima:
\begin{align}
    f(\mathbf{x})= - \sum_{i=1}^{m} \left( 2^{-d+4}\sum_{j=1}^{d}  (x^{(j)}- C_{ij})^2 + c_i\right)^{-1},
\end{align}
where $\mathbf{x}\in [0,10]^d$, $C_{d\times m}=[x_1^*,x^*_2,x^*_3,\ldots,x^*_m]$ is the set of
 local minima and $c_i=[w_1,w_2,\ldots,w_m]^T$ is the weights of corresponding local minima; 
 the smallest weight determines the global minima.
In our experiments we choose $c_i=   \{ 0.1, 0.2, 0.2, 0.4, 0.4, 0.6, 0.3, 0.7, 0.5, 0.5\}$ for 
a set of $m=10$ local minima.

\subsection{Experimental analysis}
\begin{figure}[t]
    \centering
    \includegraphics[ width=0.8\linewidth]{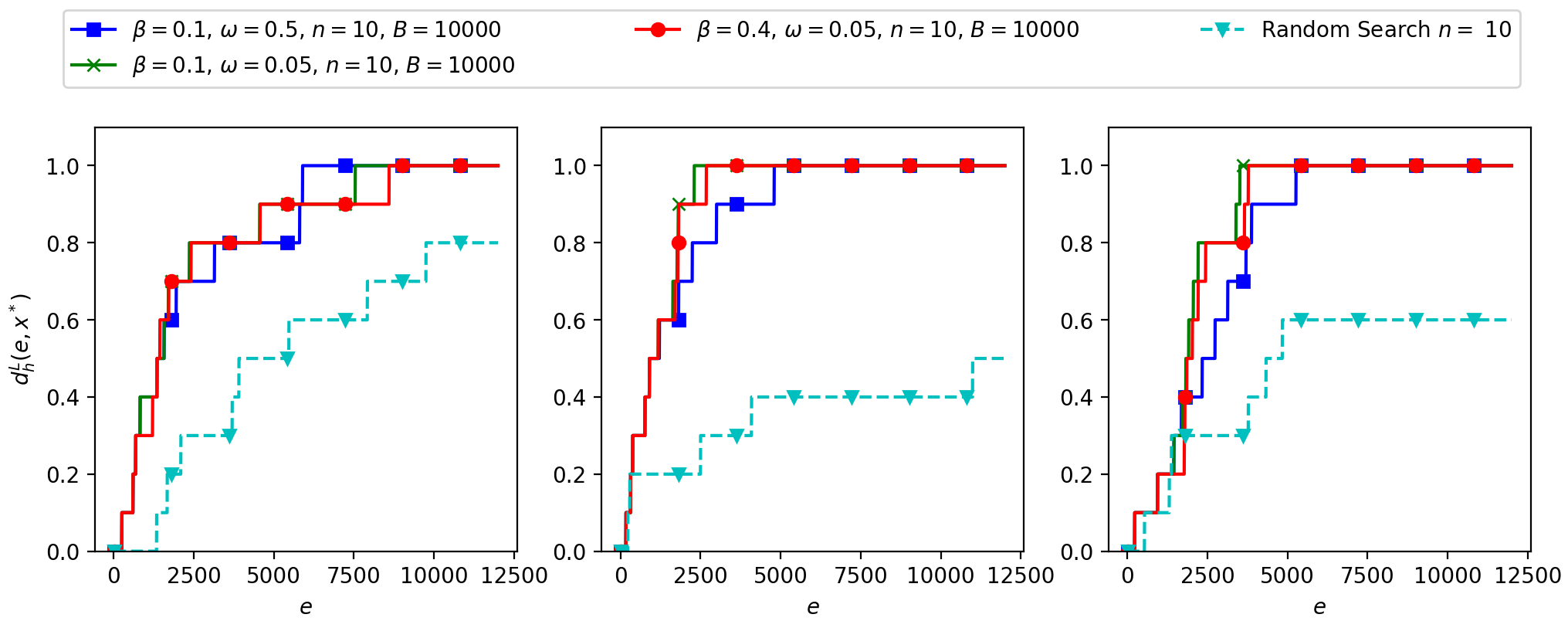}
    \caption{Data profiles for the Branin--Hoo function for finding each local minimum. $\zeta=10^{-3}$
         and $|\cP|=10$. 
     }
    \label{fig:Branin2D}
\end{figure}
\begin{figure}[t]
    \centering
    \includegraphics[ width=0.9\linewidth]{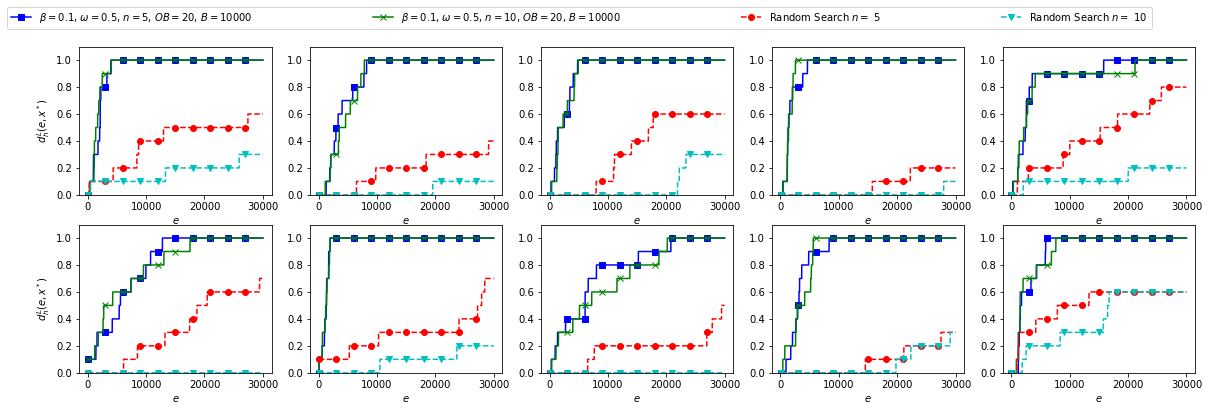}
    \caption{Data profiles for the Shekel-4D function for finding each local minimum. $\zeta=10^{-4}$ 
        and $|\cP|=10$. 
    }
    \label{fig:Shekel4D}
\end{figure}

\begin{figure}[b]
    \centering
    \includegraphics[ width=0.9\linewidth]{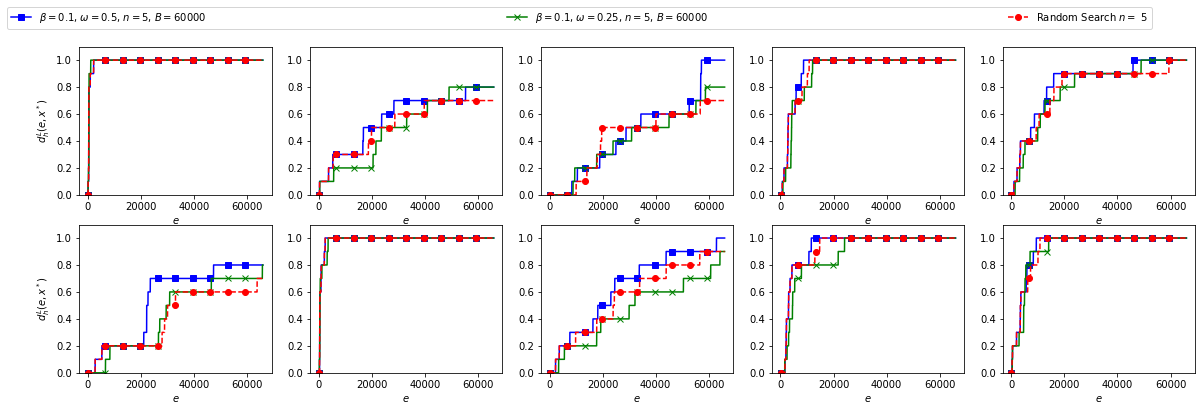}  
    \caption{Data profiles for the Shekel-10D function for finding each local minimum.
        $\zeta=10^{-4}$ and $|\cP|=10$. 
    }
    \label{fig:Shekel10D}
\end{figure}


We plot the data profiles for the Branin function ($d=2$) in Figure~\ref{fig:Branin2D}
and the Shekel $d=4$ and $d=10$ functions in 
Figures~\ref{fig:Shekel4D} and
~\ref{fig:Shekel10D}, respectively. 
We observe from the data profiles presented in each plot that MANSO outperforms the uniform
random search method in finding a point in a ball of volume $10^{-4}$
($10^{-3}$ for the Branin function) 
times  the volume of the domain $\cD$
centered at the respective true local minima. Next, we discuss the effect of MANSO hyperparameters on its performance.  \textit{Effect of $n$:} Recall that as $n$
increases the variance in estimates of the function values decrease. However,
this confidence is attained at the  expense of shedding more budget.
Consequently, the number of points evaluated by MANSO decrease. 
Hence  varying $n$ controls the trade-off between exploration and variance in function evaluation.  
We can observe
this effect by comparing the MANSO performance on Shekel-4D function
in~Figure~\ref{fig:Shekel4D}. Nonetheless, it is evident from the plots 
in~Figure~\ref{fig:Shekel4D}
that the increasing $n$ does not guarantee that finding the local minima will be faster.
\textit{Effect of $\omega$:} Recall
that $\omega$ is a hyperparameter used in Step 7 of MANSO, to terminate an LSO
run if any of its \Rev{iterates are} within $2\omega$ distance of any iterate generated
by some other LSO. Intuitively, a larger $\omega$ would result in more
termination. In Figure~\ref{fig:Branin2D}, the  effect of $\omega$ can be
observed by  comparing the green ($\omega =0.05$) and blue ($\omega =0.5$)
lines. Note that MANSO under the green experiment was able to
explore more points and thus identified 2 local minima of the  Branin function 
faster than the blue experiment. In the blue experiments, due to larger $\omega$, the number of points evaluated by MANSO is less
than the green as a large portion of the budget is used for evaluating new
sampled points. 

\subsection{Variational Parameter Optimization}~\label{sec:QAOA}

Quantum approximate optimization algorithm (QAOA) is a
hybrid algorithm that uses a parameterized trial quantum state $\psi(x)$ as
defined by the parameters $x$. (The values in $x \in \mathbb{R}^{2p}$ are
rotations or angles that parameterize $2p$ unitary operators.) What is
desired is parameters $x$ such that when the trial state is measured, the
measurement outcome corresponds to the solution of the optimization problem.
This is achieved by finding parameters $x$ that give a large expected value for
$\psi(x)^T H \psi(x)$, where $H$ is the problem Hamiltonian encoding some
classical objective $h$. Under certain conditions on the Hamiltonian $H$,
$f(x)$ must be evaluated by using a quantum computer. The search for optimal
parameters $x$ can therefore be considered as a (classical) numerical
optimization problem of the form \eqref{eq:prob_def} with $f(x) = \psi(x)^T H
\psi(x)$. The stochasticity in the objective arises from not being able to
compute the value of observable $\psi(x)^T H \psi(x)$ by using a quantum circuit
but rather having to compute the objective from a sample: $f(x) = \sum_{y_i \in
    \text{sample}}h(y_i)$.

QAOA has nontrivial performance guarantees~\cite{farhi2014quantum,
    farhi2014quantumbounded} and requires the execution of only moderately sized
quantum circuits, with the depth controlled by the number of steps $p$. For
these reasons, QAOA is an especially promising candidate algorithm for
demonstrating quantum advantage on near-term quantum computers. Yet, the
quality of the solution produced by QAOA depends critically on the quality of
the parameters $x$ used by the algorithm. Identifying such parameters is 
difficult because the
objective landscape is highly nonconvex with many local minima with poor
objective values~\cite{zhou2018quantum,Shaydulin2019}.
Figure~\ref{fig:2d_qaoa} shows
an example contour
plot with $p=1$.
While nonglobal optima are not
necessarily of interest in the QAOA problem setting, we consider the difficulty
of finding a global optimum to be a considerable test of our MANSO implementation. 

We consider the problem of using QAOA to find the maximum cut on the
Petersen graph with \Rev{a} depth of $p=5$, that is, $d=10$. The global optimal value
for this problem is $-12$. The performance of MANSO to identify maximum cut
with three local solvers (ASTRO-DF~\cite{Shashaani2018},
BOBYQA~\cite{cartis2018improving}, 
and Snobfit~\cite{huyer2008snobfit}) are summarized in
Figure~\ref{fig:2d_qaoa}. Moreover, 
we also considered non-MANSO global \Rev{optimizers} such as Bayesian optimization~\cite{BOFernando}
to solve a deterministic version of \Rev{the} QAOA problem. However, the method was significantly slow due to 
large matrix computation and produced the best candidate global minima with \Rev{a} value -11.74 only after 
 5000 evaluations. Consequently, we are not comparing MANSO with other approaches such 
 as Bayesian optimization.
For each local \Rev{solver,} we run MANSO on the MAXCUT problem
with a budget of 150,000 function evaluations and we check for the termination
condition in step 7 of MANSO after 500 function evaluations \Rev{have been performed} by the local search method. 
We also fix $n=5$, $\omega=0.01$, $\beta=0.1$
and  $\tau=0.01$. We repeat each experiment with a given local
solver 20 times and plot the range of best function value identified in
Figure~\ref{fig:2d_qaoa}. 
Although BOBYQA is designed to solve
deterministic problems, we use
it for this stochastic problem as it has been reported
that it empirically performs well on problems
with stochastic noise~\cite{cartis2018improving}. In particular, it is evident from  our experimental
result in Figure~\ref{fig:2d_qaoa} (right) too that BOBYQA performance is
competitive with other stochastic solvers. However, \Rev{we note} that BOBYQA
failed on 10 out of 20 experiments as it produced singular Hessian matrices of
\Rev{the} noisy QAOA objective.

\begin{figure}
    \centering
    \includegraphics[scale = 0.45]{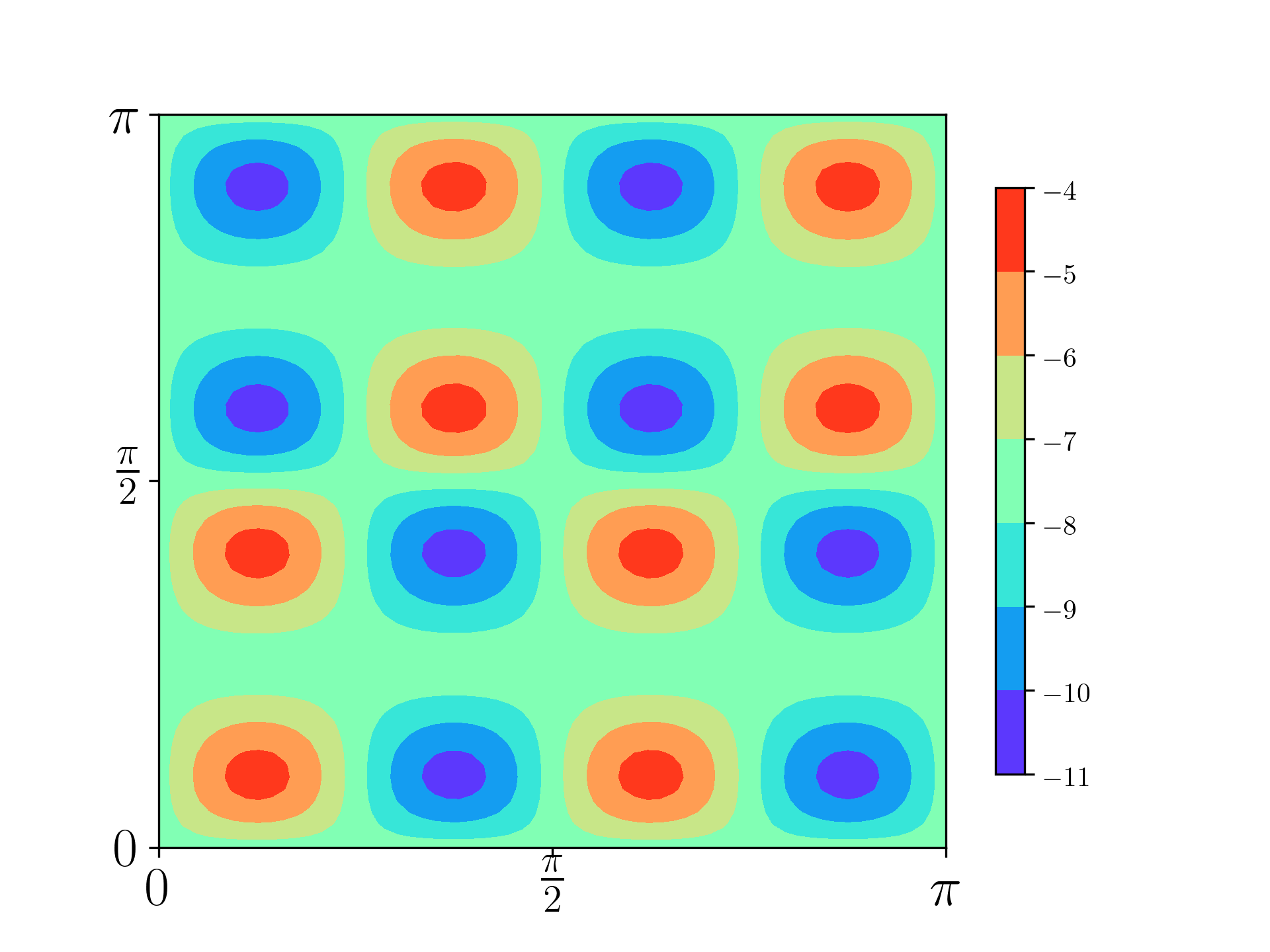}
    \includegraphics[scale = 0.55]{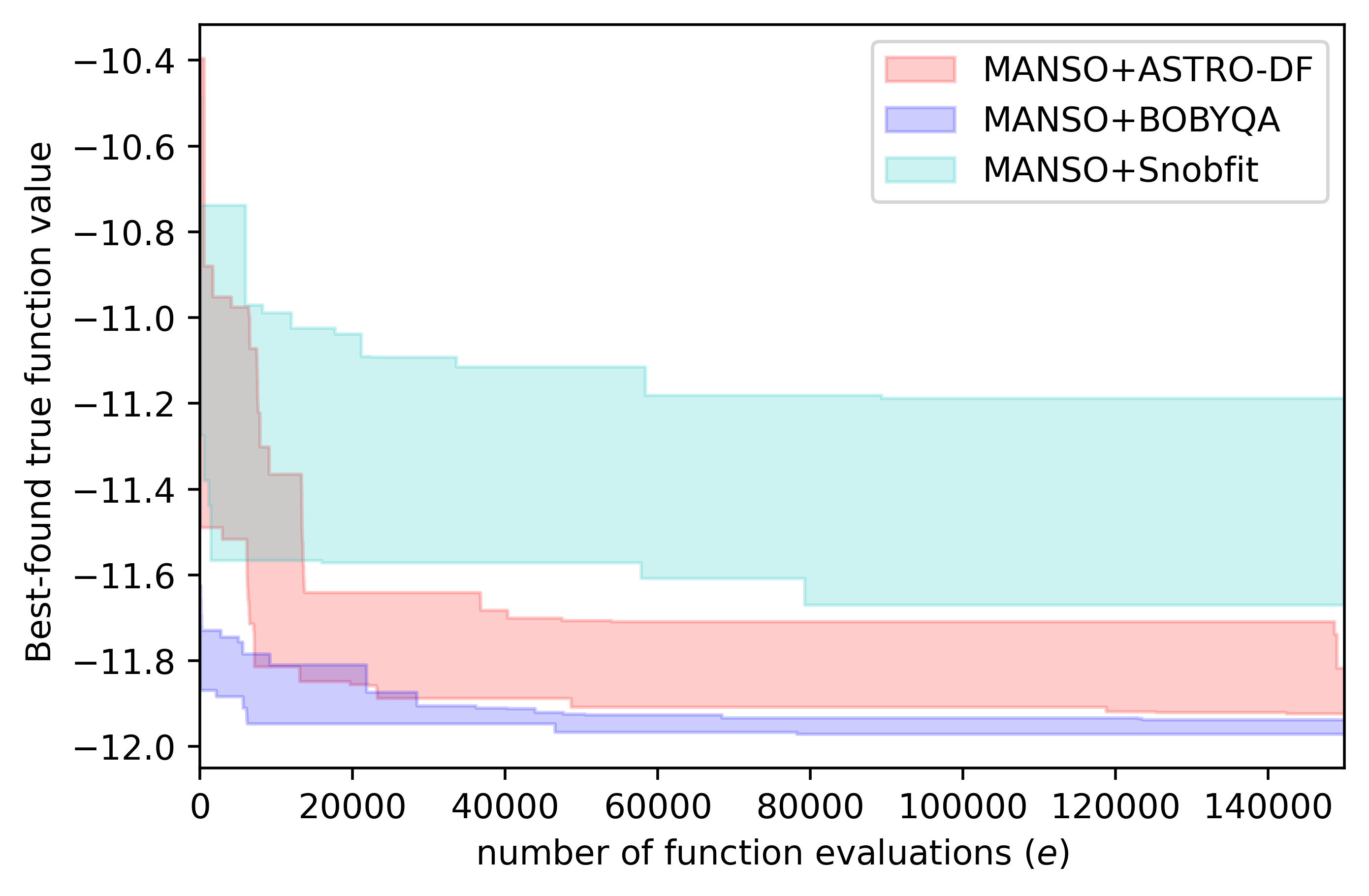}
    \caption{(left) Example landscape for MAXCUT on the Petersen graph with $p=1$. (right) Performance of MANSO in identifying optimal parameters within QAOA with various local solvers.\label{fig:2d_qaoa}} 
\end{figure}

\section{Conclusion}
\label{sec:concl}
We propose the MANSO algorithm to identify all the local minima of a
stochastic nonconvex function. We construct an efficient scheme to judiciously determine when to 
start a local stochastic optimization run from a sampled point in a compact
search domain. We show that under that MANSO starts only finitely many
local stochastic optimization runs. We also show that MANSO identifies all the
local minima asymptotically with high probability, given that the local stochastic
optimization method is guaranteed to converge to a local minimum with high
probability. MANSO's theoretical guarantees also require that the sequence of
iterates generated from the local stochastic search started in a domain of
attraction and cannot leave that domain with high probability.
(Certainly, this is a restrictive assumption for a stochastic optimization
method, but it is analogous to the assumptions in the foundational MLSL
work~\cite{RinnooyKan1987, RinnooyKan1987ii}.) Our experimental results show that
MANSO can display strong performance even when coupled with a local optimization
that does not satisfy such a restrictive assumption. The assumption that there
are no flat regions in the true objective function may be removed by using the
techniques developed in~\cite{Locatelli1998} for a multistart algorithm MLSL for
deterministic nonconvex objectives.

Furthermore, we demonstrate the efficacy of our algorithm on two benchmark
problems with dimensions ranging from 2 to 10, and we compare the performance
with that of a uniform random search
method. We also use MANSO to find the global minima of a highly nonconvex
10-dimensional Peterson graph. We aim to apply MANSO to more complex and 
higher-dimension benchmark functions and application problems as part of our future
work. Similar to~\cite{Larson2018}, an asynchronously parallel version of MANSO
can be developed to improve its computational performance.

\section*{Acknowledgments}
This work was supported by the U.S. Department of Energy, Office of Science,
     Office of Advanced Scientific Computing Research, Accelerated Research for
Quantum Computing program, under contract number DE-AC02-06CH11357.

\bibliographystyle{spmpsci}      
\bibliography{../../bibs/refs}   
\framebox{\parbox{.90\linewidth}{\scriptsize The submitted manuscript has been created by
        UChicago Argonne, LLC, Operator of Argonne National Laboratory (``Argonne'').
        Argonne, a U.S.\ Department of Energy Office of Science laboratory, is operated
        under Contract No.\ DE-AC02-06CH11357.  The U.S.\ Government retains for itself,
        and others acting on its behalf, a paid-up nonexclusive, irrevocable worldwide
        license in said article to reproduce, prepare derivative works, distribute
        copies to the public, and perform publicly and display publicly, by or on
        behalf of the Government.  The Department of Energy will provide public access
        to these results of federally sponsored research in accordance with the DOE
        Public Access Plan \url{http://energy.gov/downloads/doe-public-access-plan}.}}

\clearpage
\appendix

\section{Proofs}
Below are the proofs of \lemref{RK} and \lemref{RK2}.

\begin{proof}[Proof of Lemma~\ref{lem:RK}]
	First, consider the sets
	\begin{align}
		\cG(a;r;n)&:=\left \{ x \in  \cD:  \| x- a\|\leq r \text{ and }  [f(x)-f(a)] < 
		\e_n(x;a) \right \} \text{ and }
		\label{eq:L0}
		\\
		\cC(a;r)&:= \left\{x\in \cD: \| x- a\|\leq r \text{ and } 
		[f(x)- f(a)] < 0 \right\}.
		\label{eq:L01}
	\end{align}
	Now, observe that using Chebyschev's inequality and \assref{Fhat},
	for all $x \in \cG(a;r;n)$,
	\begin{align} 
		\nonumber P_{\xi} &\left(\hat f_n(x) - \hat f_n(a) 
		> \e_n(x;a) \right)
		\\
		\nonumber
		&\leq P_{\xi}\left( \left| [\hat f_n(x)
		-f(x)] - [\hat f_n(a)-f(a)] \right| > \e_n(x;a) - [f(x)-f(a)] \right)
		\\
		\nonumber
		&\leq \frac{1}{(\e_n(x;a)-[f(x)-f(a)] )^2}\E[\xi]{\left([\hat f_n(x)
			-f(x)] - [\hat f_n(a)-f(a)]\right)^2} 
		\\ 
		\nonumber 
		&= \frac{1}{(\e_n(x;a)-[f(x)-f(a)]
			)^2}\Bigg(\Var[\xi]{\hat f_n(x)}+ \Var[\xi]{\hat f_n(a)} 
		- 2\Cov[\xi]{\hat 
			f_n(x),\hat f_n(a)}\Bigg) 
		\\ 
		&= \frac{\Var[\xi]{\hat f_n(x)-\hat f_n(a)}}{(\e_n(x;a)-[f(x)-f(a)] )^2}. 
		\label{eq:L1}
	\end{align} 
	
	Since $x \in \cG(a;r;n)$ implies $\e_n(x;a)> [f(x)- f(a)] $, therefore
		\begin{align} 
			\nonumber
			&\left\{x \in \cD: \frac{\Var[\xi]
				{\hat f_n(x)-\hat f_n(a) 
			}}{(\e_n(x;a)-[f(x)-f(a)] )^2}
			< \beta \right\} 
			\\
			&= \left\{ x \in \cD: f(x)- f(a) < \e_n(x;a) -
			\sqrt{\beta^{-1}{\Var[\xi]{\hat f_n(x)-\hat f_n(a) }}} \right\}. 
			\label{eq:L02}
		\end{align} 
	
	Next, recall the definition of $\A{a}{r}{n}{\beta}$, 
	and observe that equation~\eqref{eq:L1} and~\eqref{eq:L02} together imply that 
	\begin{align}\label{eq:L3}
		\nonumber
		\bigg\{x\in \cD: \| x- a\|\leq r \text{ and } 
		f(x)- f(a) < \e_n(x;a) &
		\\- \sqrt{\beta^{-1}{\Var[\xi]{\hat f_n(x)-\hat f_n(a) }}} 
		\bigg\} &\subseteq \A{a}{r}{n}{\beta},
	\end{align}
	for all $x \in \cG(a;r;n)$. Since $\e_n(x;a) = \sqrt{\frac{\Var[\xi]{\hat 
				f_n(x)-\hat f_n(a) }}{\beta}}$, equation~\eqref{eq:L3} 
	implies that
	\begin{align} 
		\cC(a;r) \cap \cG(a;r;n)  \subseteq \A{a}{r}{n}{\beta} \cap \cG(a;r;n).
		\label{eq:L3a}
	\end{align}
	Observe that  $\cC(a;r)\subseteq \cG(a;r;n)$ and $\A{a}{r}{n}{\beta} \cap 
	\cG(a;r;n) \subseteq \A{a}{r}{n}{\beta}  $; Therefore it follows 
	from~\eqref{eq:L3a} that 
	\begin{align} 
		\cC(a;r)  \subseteq \A{a}{r}{n}{\beta}.
		\label{eq:L4}
	\end{align}
	Recall that $\cT_{\omega}$ is the union of $\omega$-radius balls centered at 
	stationary points of $f$ in $\cD$. Now consider $a \in
	\cD \setminus (\cQ_{\tau}\cup \cT_{\omega}) \}$. Define the set \[ \cE(a;r;\rho):=
	\left\{x\in \cD: \| x- a\|\leq r \text{ and } \nabla f(a)^T(x-a) + \frac{1}{2} \rho
	r^2 \leq 0 \right\}, \]
	where 
	$\rho$ is the largest eigenvalue of $\nabla^2 f(x)$ for $x \in \cD$. 
	Using 
	the Taylor expansion of $f$ around $a$, we know that for
	all $x\in \cD$, with $\|x-a\|\leq r$, 
	there exists a $\theta\in [0,1]$ such
	that 
	\begin{align} 
		f(x) - f(a) &= \nabla f(a)^T(x-a) + \frac{1}{2}(x-a)^T \nabla^2 f(a + \theta(x-a) )
		(x-a) 
		\label{eq:l1}.
	\end{align}
	For ease of reference, let $H = \nabla^2 f(a + \theta(x-a) ) = H$ and $v = x-a$.
	
	Since $f$ is twice continuously differentiable by~\subassref{prob}{c2}, its  
	Hessian is always real and symmetric, satisfying
	$v^T H v \leq \rho v^Tv$ for all $v \in \mathbb{R}^d$. It follows
	from~\eqref{eq:l1} that
	\begin{align}
		f(x) - f(a) &\leq \nabla f(a)^T(x-a) + \frac{1}{2} \rho (x-a)^T (x-a) 
		\leq \nabla f(a)^T(x-a) + \frac{1}{2} \rho r^2.
		\label{eq:l2}
	\end{align} 
	Equation~\eqref{eq:l2} implies that $\cE(a;r;\rho)
	\subseteq \cC(a;r) \subseteq \A{a}{r}{n}{\beta}$.
	For a given $a\in \cD$, \( m( \B{a}{r} ) \leq r^d \frac{\pi^{d/2} }{ \fGamma(1+ d/2)} \), 
	where $\fGamma(\cdot)$ is the gamma function. 
	Now using the lower bound on $m(\cE(a;r;\rho))$ derived in Lemma 7 of~\cite{RinnooyKan1987}, 
	we obtain 
	\begin{align}
		\lim_{r \to 0} \frac{m(\A{a}{r}{n}{\beta})}{m ( \B{a}{r} )} \geq 
		\lim_{r \to 0} \frac{m(\cE(a;r;\rho))}{m ( \B{a}{r} )} \geq  \lim_{r
			\to 0} \frac{1}{2} - \frac{\pi^{-d/2}}{2p\fGamma(1 + \frac{d-1}{2})} 
		\frac{\rho r}{2} = \frac{1}{2},
		\label{eq:L6}
	\end{align} 
	where $p = \min\{ \|\nabla f(a)\|, a \in
	\cD\setminus (\cQ_{\tau}\cup \cT_{\omega}) \}$ for any $k \in \N$. (Note that $p
	> 0$ by construction.) 
	Therefore~\eqref{eq:L6} implies that for all $a \in \cD 
	\setminus (\cQ_{\tau}\cup \cT_{\omega}) $
	\[ \lim_{r
		\to 0} \frac{m(\A{a}{r}{n}{\beta})}{m ( \B{a}{r} )} \geq \frac{1}{2}.  \] 
	\flushright \qed
\end{proof}

\begin{proof}[Proof of Lemma~\ref{lem:RK2}]
	For any $a\in$ $S_k\setminus (
	\cQ_{\tau}\cup \cT_{\omega})$ and using the fact that $S_k$ is sampled uniformly, for
	a vanishing sequence $\{r_k\}$ the \lemref{RK} implies 
	that there exists a $k_0\geq 1$ such that for all $k\geq k_0$, 
	\begin{align}
		\nonumber
		P[\{ t_k' > 0\}] &\leq \bigcup_{a \in  S_k\setminus (
			\cQ_{\tau}\cup \cT_{\omega})}  \left(1-\frac{m(\A{a}{r_k}{n}{\beta})}{m(\cD)} \right)^{|S_k|-1} \\
		\nonumber
		&\leq |S_k\setminus (
		\cQ_{\tau}\cup \cT_{\omega})| \left(1-\frac{m(\A{a}{r_k}{n}{\beta})}{m(\cD)} \right)^{|S_k|-1} \\
		&\leq |S_k| \left(1-\frac{1}{2} \frac{m(\B{a}{r_k})}{m(\cD)} \right)^{|S_k|-1}
		\label{eq:eq2},
	\end{align}
	where the first inequality bounds the probability that  for any sampled
	point in $S_k\setminus ( \cQ_{\tau}\cup \cT_{\omega})$ none of the remaining
	sampled points are in  $\A{a}{r_k}{n}{\beta}$ (see~(\textbf{S1})   of~\tabref{CON}).
	The second inequality follows from Boole's inequality. The last
	inequality in~\eqref{eq:eq2} is due to~\lemref{RK} and the fact  that
	$|S_k\setminus ( \cQ_{\tau}\cup \cT_{\omega})|<  |S_k|$.  
	
	Recall that for any $a\in \cD$ and $r_k$ sufficiently small such that
	$\B{a}{r_k} \subseteq \cD$, 
	$m( \B{a}{r_k} ) = r_k^d \frac{\pi^{d/2} }{ \fGamma(1+ d/2)}$. 
	Combined with this fact, for any $\sigma>0$ and choosing a sequence 
	$r_k = \frac{1}{\sqrt{\pi}} \sqrt[d]{\fGamma(1+{d/2}) m(\cD) \sigma \frac{\log|S_k|}{|S_k|}}$,
	we have 
	\[
	P[\{ t_k' > 0\}] 
	\leq |S_k| \left(1- \frac{1}{2} \frac{m(\B{a}{r_k})}{m(\cD)} \right)^{|S_k|-1}
	= |S_k|\left(1-\frac{\sigma}{2} \frac{\log|S_k|}{|S_k|} \right)^{|S_k|-1}. 
	\]
	Since $|S_k|\to \infty$ as $k\to \infty$, the fact that
	$e^{-\frac{\sigma}{2} \frac{\log|S_k|}{|S_k|}}\geq 1-\frac{\sigma}{2} \frac{\log|S_k|}{|S_k|}$
	implies 
	\begin{align}
		P[\{ t_k' > 0\}] \leq |S_k|\left(e^{-\frac{\sigma}{2} \frac{{(|S_k|-1)} \log|S_k|}{|S_k|}} \right) = |S_k|^{1-\frac{\sigma}{2}\frac{{(|S_k|-1)} }{|S_k|} }= O(|S_k|^{1-
			\frac{\sigma}{2} }).
		\label{eq:L21}
	\end{align}
	\flushright \qed
\end{proof}

\section{Shekel with $d=6$ and $d=8$}\label{app:1}

\begin{figure}[h]
    \centering
    \includegraphics[ width=1\linewidth]{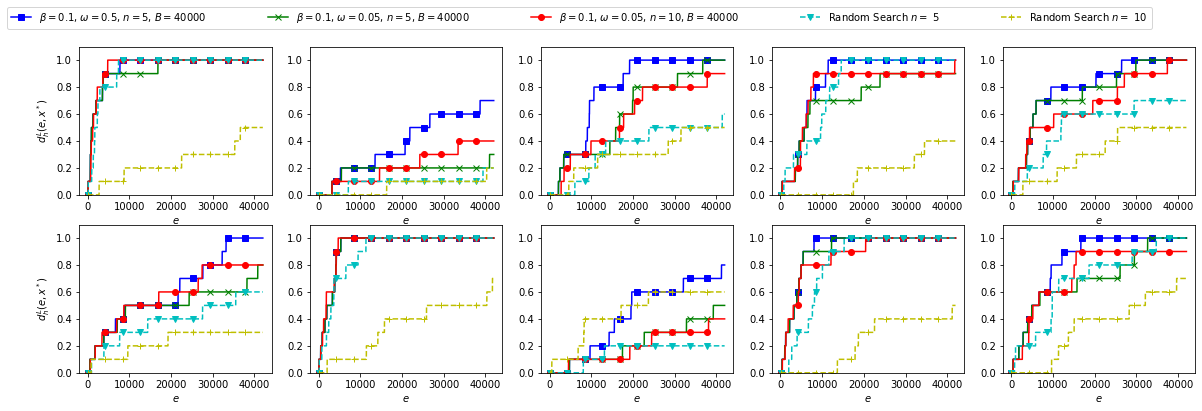}
    \caption{Data profiles for the Shekel-6D  function for finding each local minimum.
        $\zeta=10^{-4}$ and $|P|=10$. }
    \label{fig:Shekel6D}
\end{figure}

\begin{figure}[h]
    \centering
    \includegraphics[ width=1\linewidth]{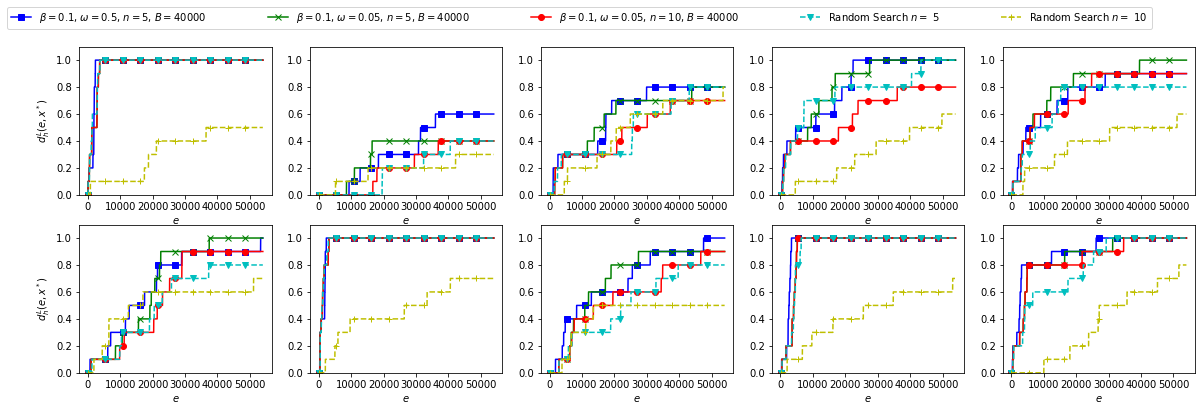}
    \caption{Data profiles for the Shekel-8D  function for finding each local minimum.
        $\zeta=10^{-4}$ and $|P|=10$. }
    \label{fig:Shekel8D}
\end{figure}


\end{document}